\documentclass[10pt]{amsart}
\usepackage{fullpage}

\usepackage{amsmath, amsthm, amssymb, graphicx}

\newcounter{citedtheorems}

\newtheorem{defn}{Definition}[section]
\newtheorem{theorem}[defn]{Theorem}
\newtheorem*{theorem-abs1}{Theorem \ref{ind-theorem}}
\newtheorem*{theorem-abs2}{Theorem \ref{a23}}
\newtheorem*{theorem-abs3}{Theorem \ref{ind-new}}
\newtheorem*{theorem-abs4}{Theorem \ref{m1}}
\newtheorem*{theorem-acor}{Conclusion \ref{k26}}
\newtheorem{thm-lit}[citedtheorems]{Theorem}
\newtheorem{defn-lit}[citedtheorems]{Definition}
\newtheorem{fact-lit}[citedtheorems]{Fact}
\newtheorem{fact}[defn]{Fact}
\newtheorem{cor}[defn]{Corollary}

\newtheorem{conv}[defn]{Convention}
\newtheorem{claim}[defn]{Claim}

\newtheorem{disc}[defn]{Discussion}

\newtheorem{expl}[defn]{Example}

\newcommand{\br}{\vspace{2mm}}

\newcommand{\ml}{\mathcal{L}}
\newcommand{\tlf}{\trianglelefteq}

\newcommand{\de}{\mathcal{D}}

\newcommand{\trv}{\mathbf{t}}

\newcommand{\vp}{\varphi}

\title{Notes on the stable regularity lemma}

\author{M. Malliaris and S. Shelah}
\thanks{\emph{Thanks:} Malliaris was partially supported by NSF DMS-1553653 and by a Minerva Research Foundation membership at the IAS. 
Shelah was partially supported by %European Research Council 
ERC grant 338821.   Both authors thank 
NSF grant 1362974 (Rutgers) and ERC 338821.  This is paper E98 in Shelah's list. 
These notes especially benefitted from lectures in the Chicago REU 
over several summers.  
We thank 
E.  Bajo  
and C. Terry for helpful comments on versions of this manuscript. } 

\address{Department of Mathematics, University of Chicago, 5734 S. University Avenue, Chicago, IL 60637, USA}
\email{mem@math.uchicago.edu}

\address{Einstein Institute of Mathematics, Edmond J. Safra Campus, Givat Ram, The Hebrew
University of Jerusalem, Jerusalem, 91904, Israel, and Department of Mathematics,
Hill Center - Busch Campus, Rutgers, The State University of New Jersey, 110
Frelinghuysen Road, Piscataway, NJ 08854, USA.}
\email{shelah@math.huji.ac.il}
\urladdr{http://shelah.logic.at}

\begin{document}

\begin{abstract} 
This is a short expository account of the regularity lemma for stable graphs proved by the authors, with some comments on the 
model theoretic context, written for a general logical audience.    
\end{abstract}

\maketitle

Some years ago, we proved a ``stable regularity lemma'' showing essentially that Szemer\'edi's regularity lemma behaves much better if the graphs in question do not contain large half-graphs \cite[Theorem 5.18]{MiSh:978}.   Since that time, it has been a pleasure to 
see the work which has grown out from that theorem, with 
various interesting extensions, further developments, and new directions worked out by many different colleagues.  
Nonetheless, it seems the clear `picture' of the original proof has not necessarily been widely communicated. Perhaps having a short exposition 
available may help inspire further interactions and applications. 

So in these brief expository notes we give a short overview of the proof itself and the model-theoretic ideas behind the proof.  
Recall that our story begins with: 

\begin{thm-lit}[Szemer\'edi's regularity lemma, 1978] 
For every $\epsilon > 0$ there is $N(\epsilon)$ s.t. every finite graph $G$ may be partitioned into $m$ classes
$V_1 \cup \cdots \cup V_m$ where $m \leq N$ and:

\begin{itemize}
\item all of the pairs $V_i, V_j$ satisfy $||V_i|-|V_j|| \leq 1$.
\item {all but at most $\epsilon m^2$} of the pairs $(V_i, V_j)$ are $\epsilon$-regular.
\end{itemize}
\end{thm-lit}

It was known by work of Gowers that the bound $N$ on the number of pieces is very large as a function of $\epsilon$ \cite{gowers}.  Regarding whether the irregular pairs can be eliminated, 
in \cite{ks}, section 1.8, Koml\'os and Simonovits write: ``\emph{The Regularity Lemma 
does not assert that all pairs of clusters are regular.  In fact, it allows $\epsilon k^2$ pairs to be irregular.  For a long time it was not 
known if there must be irregular pairs at all. It turned out that there must be at least $ck$ irregular pairs.}'' They continue: 
``\emph{Alon, Duke, Leffman, R\"odl and Yuster \cite{alon} write: ‘In \cite{sz1} the author raises the question if the assertion of the lemma holds when we do not allow any irregular pairs in the definition of a regular partition. This, however, is not true, as observed by several researchers, including 
L. Lov\'asz, P. Seymour, T. Trotter and ourselves. A simple example showing irregular pairs are necessary is a bipartite graph with vertex classes 
$A = \{ a_1,...,a_n \}$ and $B = \{ b_1,...,b_n \}$ in which $a_i b_j$ is an edge iff $i \leq j$}.’ 
[footnote: \emph{This important graph is called the half-graph.}] ”

The stable regularity lemma will show that half-graphs \emph{characterize} existence of irregular pairs by proving that in the 
absence of long half-graphs, there is a much stronger regularity lemma, which among other things, has no irregular pairs. 
At the time, the idea that there might be better regularity lemmas on certain sub-classes of graphs was not new:  
it was already known \cite{alon}, \cite{l-s} that assuming bounded VC dimension, the number of pieces could be 
taken to be polynomial in $\frac{1}{\epsilon}$ (though necessarily with irregular pairs).  
However, to our knowledge, the order property was not suspected 
by the combinatorial community be an indicator of a sea-change in structure. 

To a model theorist, half-graphs are an instance of the \emph{order property} for the graph edge relation. In the case of 
infinite structures, 
we know from the second author's \emph{Classification Theory}, Theorem II.2.2 that the presence or absence of the order property, {here} 
the presence or absence of infinite half-graphs, is a very strong indicator of a change in structural properties, the 
\emph{dividing line} at stability.    
One of the contributions of \cite{MiSh:978} was the idea that one might try to 
finitize some of the structure familiar from stability to prove that when half-graphs are small relative to the size of the 
finite graph, one may look for suitable finite approximations to stable behavior -- and thus build a much stronger regularity lemma.     
Note that an interesting line of work starting with C. Terry and J. Wolf \cite{tw} has since carried this idea further and into the arithmetic setting.

\section{Proof of the lemma}

In this section we review the Stable Regularity Lemma as it was proved in \cite{MiSh:978}, Section 5.   All graphs in the paper are finite. 

Given $k \in \mathbb{N}$, we say the graph $G$ is \emph{$k$-edge stable} if it contains no half-graph of length $k$. That is, there do not exist 
distinct vertices $a_1, \dots, a_k, b_1, \dots, b_k$ of $G$ such that $R(a_i, b_j)$ holds iff $i<j$. (Note that this forbids a family of graphs, not just a bipartite half-graph.)

\begin{theorem}[Stable regularity lemma, \cite{MiSh:978}] \label{t:sr}
For each $\epsilon > 0$ and $k \in \mathbb{N}$ there is $N = N(\epsilon, k)$ such that for  
any sufficiently large finite $k$-edge stable graph $G$, for some $\ell$ with $\ell \leq N$, 
$G$ can be partitioned into disjoint pieces $A_1, \dots, A_\ell$ and: 
\begin{enumerate}
\item the partition is equitable, i.e. the sizes of the pieces differ by at most $1$, 
\item \emph{all} pairs of pieces $(A_i, A_j)$ are $\epsilon$-regular, and moreover have density 
either $>1-\epsilon$ or $<\epsilon$. 
\item  $N < (\frac{4}{\epsilon})^{2^{k+3}-7}$. 
\end{enumerate} 
\end{theorem}

The strategy of \cite[\S 5]{MiSh:978} is to use the hypothesis of $k$-edge stability to divide a given graph into pieces 
which have an atomicity property called $\epsilon$-excellence. It will follow from the definition that the distribution 
of edges between any two $\epsilon$-excellent sets is highly uniform, in fact $\delta$-regular for a related $\delta$.  
So if we can get an equipartition into excellent pieces, we will have regularity for \emph{all} pairs, just by construction. 
This is exactly what the paper does: proves a regularity lemma with excellent pieces (Theorem \ref{su} below), from which 
Theorem \ref{t:sr} is easily deduced. 

After giving the proof, we will make some comments on the model theoretic ideas behind it. 

Comment on notation: we consider graphs model-theoretically, that is, 
as a set of vertices on which there is a certain symmetric irreflexive binary relation, the edge relation. This is reflected in 
writing things like (a) ``$A \subseteq G$'' to mean $A$ is an induced subgraph of $G$ corresponding to this set 
of vertices, (b) ``$b \in G$'' to mean $b$ is a vertex of $G$, and (c) writing ``size of''  a graph $G$ or of some $A \subseteq G$ to mean the number of vertices.

\begin{defn} \label{d:good}
Call $A \subseteq G$ \emph{$\epsilon$-good} if for any $b \in G$ either $| \{ a \in A : R(b,a) \}| < \epsilon|A|$ or 
$| \{ a \in A : \neg R(b,a) \} | < \epsilon |A|$. 
\end{defn}

In other words, any $b \in G$ (not necessarily outside $A$) induces a partition of $A$ into two pieces, and $\epsilon$-good means 
any such partition is strongly imbalanced: there is a \emph{majority opinion} in $A$ regarding whether to connect to $b$. We may express this by saying that there is a 
truth value $\trv = \trv(b,A) \in \{ 0, 1\}$ and that for all but $<\epsilon|A|$ of the elements of $A$, $R(b,a)^\trv$ holds. 
(In logical shorthand, for a formula $\vp$, $\vp^1 = \vp$ and $\vp^0 = \neg \vp$.)

Note that for any $a \in G$ and any $\frac{1}{2} > \epsilon > 0$, $\{ a \}$ is $\epsilon$-good. 
So a posteriori the definition ``$A$ is good'' can be described by saying: elements of $A$ have a majority opinion with respect to 
\emph{certain specific} good sets, namely the singletons. Of course, we can ask for majority opinions with respect to larger good sets. 
This leads naturally to the definition of \emph{excellent}, informally, that elements of $A$ have  
coherent, majority opinions with respect to \emph{all} other good sets. 

\begin{defn} \label{d:excellent}
Call $A \subseteq G$ \emph{$\epsilon$-excellent} if for any $B \subseteq G$,  if $B$ is $\epsilon$-good then either 
$|\{ a \in A : \trv(a,B) = 1 \}| < \epsilon|A|$ or $|\{ a \in A : \trv(a,B) = 0 \}| < \epsilon|A|$. 
\end{defn}

First note that if $A$ is $\epsilon$-excellent it is $\epsilon$-good.  
(When $B = \{ b \}$, $\trv(a,\{b\})$ is simply 1 if $R(a,b)$ and $0$ if $\neg R(a,b)$. 
Since any singleton set is good, the definition entails that if $A$ is excellent and $b \in G$, then either 
$| \{ a \in A : R(a,b) \} | < \epsilon |A|$ or $| \{ a \in A : \neg R(a,b) \}| < \epsilon |A|$. Thus $A$ is $\epsilon$-good.)
So existence of larger $\epsilon$-good sets under stability will be a special case of \ref{claim1}. 
%So Claim \ref{claim1} gives a fortiori the existence of nontrivial $\epsilon$-good sets in stable graphs. 
In the case of singleton sets, we will write simply $\trv(a,b)$. 

In Definition \ref{d:excellent}, as $B$ is good, any element $a \in A$ will reveal a majority opinion among elements of $B$ regarding whether to connect to $a$. 
If $A$ is an excellent set, most of the time the revealed opinion is the same. We may express this by clause (a) of 
Claim \ref{excellent-regular} below. 

It isn't obvious that large $\epsilon$-excellent subsets of a given graph should 
exist (e.g. a random graph tends not to have nontrivial $\epsilon$-good sets); 
our proof will use $k$-edge stability in a direct way. 
We will use the following definition and fact from model theory, which specialized to our case says that from edge stability we may 
infer a specific finite bound on the height of a certain tree, which %locally in this paper 
let us call a special tree.

\begin{defn}
A \emph{full special tree of height $n$} in a graph is a configuration consisting of two sequences of vertices, $\langle b_\rho : \rho \in 2^{<n} \rangle$, called 
nodes, and $\langle a_\eta : \eta \in 2^n\rangle$, called leaves, with edges satisfying the following constraint: 
given $\eta \in 2^n$ and $\rho \in 2^{<n}$, \emph{if} $\rho^\smallfrown\langle \ell \rangle \trianglelefteq \eta$ \emph{then} $R(a_\eta, b_\rho)^\ell$. 
\end{defn}

Other than the edges and non-edges which we have explicitly mentioned, anything is allowed; so as with half-graphs, asserting 
that there is no special tree of a certain height forbids a family of configurations.

\begin{expl}
Consider a special tree of height 2 with nodes $b_\emptyset, b_0, b_1$ and leaves 
$a_{00}, a_{01}, a_{10}, a_{11}$. Then the following 
edges and non-edges must occur: 
$R(b_\emptyset, a_{11})$, $R(b_\emptyset, a_{10})$,  
$\neg R(b_\emptyset, a_{01})$, $\neg R(b_\emptyset, a_{00})$, 
$R(b_1, a_{11})$, $\neg R(b_1, a_{10})$,  
$R(b_0, a_{01})$, $\neg R(b_0, a_{00})$. 
\end{expl}

The relevant fact relating edge stability to special trees is the following. 

\begin{fact}[see e.g. Hodges \cite{hodges} Lemma 6.7.9 p. 313] \label{tree-fact} \emph{ }
\begin{enumerate}
\item If the graph $G$ is $k$-edge stable, then there is no full special tree in $G$ of height $2^{k+2}-2$.
\item If $G$ contains no full special trees of height $n$, then $G$ is $2^{n+1}$-edge stable.  
\end{enumerate}
\end{fact}

\begin{conv} \label{tbc}
When $G$ is $k$-edge stable, we will say ``let $t=t(k)$ be the tree bound'' to abbreviate 
``let $t = t(k)$ be the best strict upper bound on the height of a full special tree from Fact \ref{tree-fact},'' 
so $t(k) \leq 2^{k+2}-2$.
\end{conv}

When $G$ is $k$-edge stable, we can use Fact \ref{tree-fact} to partition into $\epsilon$-excellent sets as follows. 
(Note that the same proof gives existence of $\epsilon$-good sets in the case where we take the $B_\eta$'s to be singletons.)

\begin{claim} \label{claim1}
Suppose $G$ is a $k$-edge stable graph, $t=t(k)$ is the tree bound, and $\epsilon < \frac{1}{2^t}$. Then for every 
$A \subseteq G$ with $|A| \geq \frac{1}{\epsilon^t}$ there exists an $\epsilon$-excellent subset $A^\prime \subseteq A$ with 
$|A^\prime| \geq \epsilon^{t-1}|A|$. 
\end{claim}

\begin{proof} 
By induction on $m \leq t$ let us try to choose 
$\langle A_\eta : \eta \in 2^m \rangle$ and $\langle B_\eta : \eta \in 2^m \rangle$ such that: 

\begin{enumerate}
\item[(i)] in general, for $m \geq 0$ and $\rho$ of length $m$, $B_\rho$ is an $\epsilon$-good set witnessing that $A_\rho$ is not $\epsilon$-excellent;  
\item[(ii)] for $m=0$, $A_\emptyset = A$ and $B_\emptyset$ is an $\epsilon$-good set witnessing that $A_\emptyset$ is not $\epsilon$-excellent. 
\item[(iii)] for $m > 0$ and $\eta$ of length $m-1$, 
\begin{itemize}
\item $A_{\eta^\smallfrown \langle 0 \rangle} = \{ a \in A_\eta : \trv(a, B_\eta) = 0 \}$ and 
\item $A_{\eta^\smallfrown \langle 1 \rangle} = \{ a \in A_\eta : \trv(a, B_\eta) = 1\}$
\end{itemize}
noting that both are nonempty, and in fact of size at least $\epsilon |A_\eta|$, since $B_\eta$ witnesses that $A_\eta$ is not excellent. 
(Why is $\trv(a, B_\eta)$ always defined? Because $B_\eta$ is good.) 
\end{enumerate}
If we can indeed choose the sets $B_\eta$ at each stage up to and including $m = t$, 
we arrive at a contradiction by building a special tree as follows. 
First we choose the leaves: for each $\eta \in 2^t$, let 
$a_\eta$ be any element of the set $A_\eta$, which is nonempty as it is of size $\geq \epsilon^t |A|$ and $|A| \geq \frac{1}{\epsilon^t}$. 
Next we choose the nodes. For each $m < t$, each $\rho \in 2^m$, and 
each $\eta \in 2^t$ such that $\rho \tlf \eta$, the set $U_{\eta} = \{ b \in B_\rho : R(a_\eta, b)^{1-\trv(a_\eta, B_\rho)} \}$ is small, 
of size $<\epsilon |B_\rho|$ because $B_\rho$ is $\epsilon$-good. All together $| U_\rho := \bigcup \{ U_{\eta} : \rho \tlf \eta \in 2^t \} | < 2^t \epsilon |B_\rho| 
< |B_\rho|$, with the last inequality from our assumption that $\epsilon < \frac{1}{2^t}$. Choose $b_\rho$ to be any element of 
$B_\rho \setminus U_\rho$. This constructs a special tree, which gives us our contradiction. 

Therefore we were wrong in assuming the construction could continue: one of the $A_\eta$ must have been $\epsilon$-excellent, 
and it will be a subset of $A$ of size at least $\epsilon^m |A|$ where $m = $length$(\eta) \geq t-1$. 
\end{proof}

Our plan is to partition the graph $G$ into $\epsilon$-excellent sets by induction: running Claim \ref{claim1} on the graph, 
setting aside the excellent subset $A_0$, running Claim \ref{claim1} on the remainder $A = G \setminus A_0$ to obtain $A_1$, $\cdots$ 
and iterating as far as we can until the 
leftover vertices are few enough to distribute among the excellent sets already obtained without causing much trouble. The issue to be solved is that 
we would like to end up with an equipartition, but 
as we've stated it, Claim \ref{claim1} gives us little control on the size of the excellent sets it returns. If the construction stops at 
the zero-th level the excellent set will have size $|A|$; at the first level, size $\geq \epsilon |A|$; at the second level, 
$\geq \epsilon^2 |A|$; and so on, with a lower bound of $\epsilon^{t-1}|A|$.  A first simple modification  
is to choose a short list of possible sizes in advance, as we now do.\footnote{\ref{size-seq}(b) suggests that 
%having 
%once fixed a size sequence in \ref{size-seq} and obtained a partition of $G$ into excellent sets whose sizes are all elements of this 
%sequence, 
we will aim for an equipartition into pieces of size $s_{t-1}$.} 

\begin{defn} \label{size-seq}
Call the sequence $s_0, \dots, s_{t-1}$ of natural numbers a \emph{size sequence for $\epsilon$} when:
\begin{enumerate}
\item[(a)] $\epsilon s_\ell \geq s_{\ell+1}$ for $\ell < t-2$. 
\item[(b)] $s_{t-1}$ divides all other elements of the sequence.  
\item[(c)] $s_{t-1} > t$. 
\end{enumerate}
\end{defn}

\begin{claim} \label{claim2}
Suppose $G$ is a $k$-edge stable graph, $t=t(k)$ is the tree bound, and $\epsilon < \frac{1}{2^t}$. Let $s_0, \dots, s_{t-1}$ 
be a size sequence for $\epsilon$.  
Then for every 
$A \subseteq G$ with $|A| \geq \max \{ s_0, \frac{1}{\epsilon^n} \}$ there exists an $\epsilon$-excellent subset $A^\prime \subseteq A$ with 
$|A^\prime| = s_\ell$ for some $\ell = 0, \dots, t-1$. 
\end{claim}

\begin{proof}
Just as in the proof of Claim \ref{claim1}, adding to (ii) of the inductive hypothesis the condition that at stage $m$, both 
$A_{\eta^\smallfrown \langle 0 \rangle}$ and $A_{\eta^\smallfrown \langle 1 \rangle}$ have size exactly $s_m$. 
This is handled by a simple modification to the construction. At stage $0$, let $A_\emptyset$ be any subset of 
$A$ of size $s_0$. At stage $m$, by inductive hypothesis, the set $A_\eta$ will have size $s_{m-1}$, so    
$A_{\eta^\smallfrown \langle 0 \rangle}$ and $A_{\eta^\smallfrown \langle 1 \rangle}$ 
will have size at least $\epsilon s_{m-1} \geq s_m$; if necessary throw away vertices to obtain size exactly $s_m$. 
\end{proof}

At this point, two points about excellent sets come in to play. First, that any two interact uniformly: 

\begin{claim} \label{excellent-regular}
Suppose $A \subseteq G$, $B \subseteq G$ and $\epsilon < \frac{1}{2}$ and both $A, B$ are $\epsilon$-excellent. Then:
\begin{itemize}
\item[(a)] the pair $(A,B)$ is $\epsilon$-uniform, 
meaning that there is 
a truth value $\trv = \trv(A,B) \in \{ 0, 1 \}$ and that for all but $<\epsilon |A|$ of the elements of $A$, for 
all but $<\epsilon|B|$ of the 
elements of $B$, $R(a,b)^\trv$.

\item[(b)] for $\zeta = \sqrt{2\epsilon}$, the pair $(A,B)$ is $\zeta$-regular with $d(A,B) > 1-\zeta$ or $d(A,B) <\zeta$. 
\end{itemize}
\end{claim}
\begin{proof}
(a) restates excellence, (b) is a straightforward calculation made in \cite{MiSh:978} Claim 5.17. 
\end{proof}

Second, for $\epsilon < \frac{1}{2}$, observe that $\epsilon$-excellence need not be preserved under subset, nor indeed is $\epsilon$-goodness. 
(Suppose the partition induced on $A$ by $b$ is imbalanced, with the smaller side of size $\ell$: $A$ has a subset of size $2\ell$ whose 
partition induced by $b$ is exactly balanced, so this subset clearly isn't $\epsilon$-good.) To preserve excellence, a subset should retain 
roughly proportional intersection with every partition induced by every element in the graph on $A$. 
However, we can ensure this if we do it randomly, because $k$-edge stability implies the VC dimension is also small. This is explained by the following, an instance of the Sauer-Shelah lemma. 

\begin{fact}[see \cite{Sh:c} Theorem II.4.10(4) p. 72] \label{a4}
If $G$ is a $k$-edge stable graph, then for any finite $A \subseteq G$, 
\[ | \{ \{ a \in A : R(a,b) \} : b \in G \} | \leq |A|^k \]
more precisely, $~\leq~ \Sigma_{i\leq k} \binom{|A|}{i}$. 
\end{fact}

\begin{cor} \label{st-vc}
Whenever $k$ is fixed, for any 
graph $G$ which is $k$-edge stable, any $A \subseteq G$ and any $\epsilon > 0$ 
the family 
\[ \{ \{ a \in A : R(a,b) \} : b \in G \} \cup \{ \{ a \in A : \neg R(a,b) \} : b \in G \}~~ \subseteq \mathcal{P}(A) \]
has VC dimension $\leq k+1$. 
\end{cor}

Again, informally, call a subset of $A$ a \emph{trace} if it is the intersection of $A$ 
with a neighborhood of some $b \in G$, or if it is the complement (in $A$) of the neighborhood of some $b \in G$. 
By the previous corollary there are relatively few traces. 
In order that the partition retain excellence for a related $\epsilon$, it suffices that every piece of the partition 
intersect all of the traces of $A$ in approximately the expected proportion.  

\begin{cor} \label{vc-claim}
For all $\zeta > \epsilon > 0$ and $r, k \geq 1$, there exists $M_{\ref{vc-claim}} = M(\zeta, \epsilon, r, k)$ such that if:
\begin{itemize}
\item[(a)] $A$ is a subset of a $k$-edge stable graph $G$, $|A| \geq M_{\ref{vc-claim}}$, 
\item[(b)] $A$ is $\epsilon$-excellent in $G$, and 
\item[(c)] the size of $A$ is divisible by $r$
\end{itemize}
then there exists a partition of $A$ into $r$ disjoint pieces of equal size each of which is $\zeta$-excellent.  
\end{cor}
 
\emph{About the bounds}: Recently Ackerman, Freer, and Patel \cite{AFP} have computed a bound on this quantity, in the context of proving a stable regularity lemma for hypergraphs, indeed for finite structures in arbitrary finite relational languages.   To read 
\cite{AFP} Proposition 4.5 specialized to the case of graphs, 
the number $|\ml|$ of non-equality symbols is 1, 
the  maximal arity of a relation $q_{\ml}$ is $2$, and $\hat{\tau}$ is their notation for the tree bound.

Returning to the main line of \cite{MiSh:978} \S 5, 
here is the core result of that section, essentially \cite{MiSh:978} Theorem 5.18 (with some more information about bounds). 

\begin{theorem}[Stable regularity lemma -- version with excellence] \label{su} 
For each $k \geq 1$ and $\epsilon > 0$, there are 
$N = N(\epsilon, k)$  
such that if $G$ is any sufficiently large $k$-edge stable graph   
there is a partition of $G$ into disjoint pieces 
$A_0, \dots, A_{\ell-1}$ with $\ell \leq N$ such that:
\begin{itemize}
\item[(a)] the partition is equitable, i.e. for $i,j < N$, $|A_i|$, $|A_j|$ differ by at most 1. 
\item[(b)] for each $i <N$, $A_i$ is $\epsilon$-excellent. 
\item[(c)] \emph{thus} every pair $(A_i, A_j)$ is $\epsilon$-uniform in the sense of \ref{excellent-regular}(a). 
\item[(d)] if $\epsilon < \frac{1}{2^t}$, $N \leq 4 \left(    \frac{8}{\epsilon}    \right)^{t-2}$, where $t=t(k)$ is the tree bound from $\ref{tbc}$. 
\end{itemize}
\end{theorem}

\begin{proof}
We're given $k$, {thus} the tree bound $t$ (recall that by Fact \ref{tree-fact}, $t \leq 2^{k+2}-2$).  Let $n$ be the size of $G$. 
Without loss of generality, to simplify notation, suppose $\epsilon < \frac{1}{2^t}$, and  $\epsilon$ is a fraction whose denominator is $1$. 

For the proof, we'll need $0 < \alpha < \beta < \epsilon$; for definiteness, we will use $\alpha = \frac{\epsilon}{4}$, $\beta = \frac{\epsilon}{3}$. 
\br

Let $q = \lceil \frac{1}{\alpha} \rceil \in \mathbb{N}$, so $\frac{2}{\alpha} \geq q \geq \frac{1}{\alpha}$.  Let $c$ be maximal such that 
$q^{t-1} c \in \left(   \frac{\alpha n}{2} - q^{t-1}, \frac{\alpha n}{2}   \right]$. 
Define a sequence by $s_0 = q^{t-1}c, s_1 = q^{t-2}c, \dots, s_{t-1} = c$.  This is a size sequence (\ref{size-seq}): it is integer valued and satisfies the conditions on divisibility, and since $G$ is sufficiently large (see step 4) 
 $c > t$. 

\br
\noindent \emph{Step 1}. By induction, construct a partition $\{ B_j : j < j_* \} \cup \{ B \}$ of $G$ into $\alpha$-excellent pieces each with size $s_\ell$ for some $\ell \in \{ 0, \dots, t-1 \}$, \emph{plus} a remainder $B$ of size $< s_0$.  
That is, apply Claim \ref{claim2} to $G$ to obtain $B_0$; if the remainder $G \setminus B_0$ has size $\geq s_0$, apply Claim \ref{claim2} again to obtain $B_1$; 
continue until fewer than $s_0$ elements remain. 

\br
\noindent \emph{Step 2}. As the elements $B_j$ of the partition are sufficiently large, \emph{see step 4}, we may randomly partition each of them into pieces of size $s_{t-1}$, i.e. $c$, which are $\beta$-excellent. 
Call this new partition $\{ B^\prime_i : i < i_* \} \cup \{ B \}$ -- we still have the remainder. 

\br
\noindent \emph{Step 3}. Distribute the remainder among the pieces $\{ B^\prime_i : i < i_* \}$ to obtain $\{ A_i : i < i_* \}$ where for all $i, j < i_*$, $||A_i| - |A_j|| \leq 1$. 
Since our choice of $c$ implies $s_0 \leq \frac{\alpha n}{2}$, a short calculation shows that this new partition remains $3\beta$-excellent.
(The calculation is given in \cite{MiSh:978} Claim 5.14(3), with $\beta$ here for $\epsilon^\prime$ there, and noting 
that our assumption $(*)$ implies $c > \frac{1}{\beta}$.)

\br

\noindent \emph{Step 4}. As for a lower bound on $n$, it is sufficient that  

(i) $\frac{\alpha^2 n}{4}-1 > \max \{ t, \frac{3}{\epsilon} \}$ and also 

(ii) $\frac{\alpha^2 n}{4}-1 > 
  \frac{1}{\alpha} M_{\ref{vc-claim}}(\beta, \alpha, r, k )$ for all integer values of $r$ in $ \{ \frac{\alpha n}{2q^{t-1}}-1, \dots, \frac{\alpha n}{2 q^{t-1}} \} $, it suffices to check for the largest, say $r = \lceil \frac{\alpha n}{2 q^{t-1}} \rceil$.  

Note that by our construction, $c \geq (\frac{\alpha n}{2} - q^{t-1})(q^{t-1})^{-1} = \frac{\alpha n}{2 q^{t-1} } -1 \geq \frac{\alpha^2 n}{4}-1$. 
So the assumption (i) says that $c$ i.e. $s_{t-1}$ is sufficiently large to get a size sequence and to run steps 1 and 3, and condition (ii) says that all the pieces in the partition are large enough to admit a random partition in step 2. 

\br
\noindent \emph{Step 5}. 
To bound the number of pieces: first note that by our choice of $c$, $\frac{\alpha n}{2} - q^{t-1} < q^{t-1} c$, so $\frac{\alpha n}{2 q^{t-1}} -1 < c$, so $\frac{\alpha n }{4 q^{t-1}} < c$.  
Note also that by choice of $q$, $q \leq \frac{2}{\alpha}$. 
Then we can bound $\frac{n}{c}$ by: $\frac{n}{c} \leq (n)\left( \frac{\alpha n}{4 q^{t-1}}\right)^{-1} = \frac{4 q^{t-1}}{\alpha} \leq  \frac{4 (\frac{2}{\alpha})^{t-1}}{\alpha} = 4 \left(  \frac{2}{\alpha}   \right)^{t-2}$. 
In terms of $\epsilon$, this is $4 \left(  \frac{8}{\epsilon}    \right)^{t-2}$. 
\end{proof}

\begin{disc}
Often such lemmas are stated with an input $m$ corresponding to a lower bound on $\ell$. It should be clear that 
by shrinking $\epsilon$, one can ensure the minimum number of pieces is as large as desired. 
\end{disc}

\begin{disc}
One could easily convert the appearance of $t$ in \ref{su}(d),(e) into a bound stated in terms of $k$ using Fact \ref{tree-fact}; 
but in cases where $t$ is computable directly from the graph $($perhaps it is much less than 
$2^{k+2}-2$$)$ the present form gives more information. 
\end{disc}

We've done the work needed for Theorem \ref{t:sr} (i.e. Conclusion 5.19 of \cite{MiSh:978}). 

\begin{proof}[Proof of Theorem \ref{t:sr}]
Apply Theorem \ref{su} with $k$ and $\frac{\epsilon^2}{2}$, and since the statement does not mention $t$, replace $k$ by $2^{k+2}-2$.  
\end{proof}

\section{A spectrum of regularity lemmas}

Section 5 is the final section of \cite{MiSh:978}. What were the aims of the rest of the paper? 

As the plural in the title suggests, the main thread of the paper investigates the structure of stable graphs by proving a spectrum of 
Regularity Lemmas, capturing different aspects of stability. 
For the first, inspired by the fact that infinite models of stable theories contain large indiscernible sets, 
we prove that for finite stable graphs or hypergraphs (indeed, in any finite stable relational structure, suitably defined) 
one can extract much larger indiscernible sets than expected from Ramsey's theorem, of size $n^c$ rather than $\log n$ for 
$c$ depending on the set of relations and their `stability,'  as measured by rank.  (This was well exposited, in the case of 
graphs, in \cite{mt}, where it found a nice application.) By iteratively using this theorem, 
one then can, with some additional care, build a first regularity lemma for stable graphs in which all pieces are cliques or 
independent sets of size $n^c$, plus a remainder, though necessarily the number of pieces grows with the size of the graph. 
The second and third regularity lemmas in some sense progressively relax the `uniformity' conditions on the pieces until 
arriving at the fourth, the stable regularity lemma described above, in which the pieces are now `only' approximately 
uniform, i.e. $\epsilon$-excellent, at the gain of the number of pieces no longer growing with the size of the graph. 

We may note that for a structure $M$ and infinite $A \subseteq M$ and ultrafilter $\de$ on $A$, there is an average type 
$\operatorname{Av}(A, \de)$. For stable $T$ we know that for suitable sets (so-called indiscernible) the filter is degenerated: 
the co-finite sets are enough. %More in this line are convergent sets, weakening stability. 
The notions of $\epsilon$-good and $\epsilon$-excellent are finitary analogues. 

The reader may wonder: the importance of both stability theory and the regularity lemma were independently well understood 
by the early eighties. Why did they come together some thirty-odd years later? As we have written elsewhere, to our knowledge 
the first connections of Szemer\'edi regularity to model theory came in the context of thinking about the relation of finite and 
infinite combinatorics necessary to understand Keisler's order \cite{mm-thesis}, \cite{mm3}.   It was not by mistake that 
that \cite{MiSh:978}, the first joint paper of the authors, came at the beginning of our joint work on Keisler's order.  
Indeed, ideas from regularity play a certain, perhaps more hidden role in our recent discovery \cite{MiSh:1167} that Keisler's order 
has the maximum number of classes, continuum many (we may refer the interested reader to the open problems section in that paper). As our knowledge of this order develops, our understanding of the 
finite is also changing.

\br

\end{document}